    \documentclass[preprint,12pt]{elsarticle}
    
    \usepackage{amssymb}
     \usepackage{amsthm}
    \usepackage{amsfonts}
    \usepackage{amsmath,amssymb}
    \usepackage[applemac]{inputenc}
    \usepackage{color}
    \usepackage{color, colortbl}
    \usepackage{amsmath}
    \usepackage{amssymb}
    \usepackage{bbm}
    \usepackage{graphicx}
    \usepackage{tikz}
    \usepackage{geometry}
    \usetikzlibrary{arrows}
    \usepackage{amsfonts}
    \usepackage{mathrsfs}
    \usepackage{amsmath,amssymb}
    \usepackage[applemac]{inputenc}
    \usepackage{color}
    \usepackage{amsmath}
    \usepackage{amssymb}
    \usepackage{graphicx}
    \usepackage{tikz}
    \newtheorem{theorem}{Theorem}[section]
    \newtheorem{lemma}[theorem]{Lemma}

    \newtheorem{example}[theorem]{Example}

    \newtheorem{remark}[theorem]{Remark}

    \journal{Journal of Differential Equations }
    
\begin{document}
    
    \begin{frontmatter}
    
    
    
    \title{Long-time behaviors of some stochastic differential equations driven by L\'{e}vy  noise}
    
     \author[label1]{I. Orlovskyi}
    \affiliation[label1]{organization={Department of Mathematical Analysis and Probability Theory, National  University of Ukraine "Igor Sikorsky Kyiv Polytechnic Institute"},
               addressline={Beresteyskyi Ave, 37},
                 city={Kyiv},
                postcode={03056},
                 country={Ukraine}}
    
    
     \author[label2]{F. Proske}
    \affiliation[label2]{organization={Department of Mathematics, University of Oslo},
               addressline={Blindern},
                 city={Oslo},
                postcode={0316},
                 country={Norway}}
     \author[label1,label2] {O. Tymoshenko}

    \begin{abstract}
    Using key tools such as It\^o formula for general semi-martingales, moments estimates for L\'{e}vy-type stochastic integrals and properties of regular varying functions we find conditions under which solutions of stochastic differential equation with jumps are almost sure asymptotically equivalent  nonrandom function with $t\to \infty$.
    
    \end{abstract}
    
    
    
    
    \end{frontmatter}
    
    
    \section{Introduction}
    The study of asymptotic properties of stochastic differential equations (SDEs) has a special place in the modern theory of random processes. Gikhman and Skorokhod \cite{GS}, Keller et al. \cite{KKR}, and later Buldygin et al. \cite{BKST}
    considered the exact order of growth of solutions of homogeneous SDEs  and found conditions, under which these solutions are asymptotically
    equivalent, as $t\to \infty$, to solutions of ordinary differential equations (ODEs). Closer investigations for multidimentional SDE are done by Samoy\v{i}lenko and Stanzhitskiy \cite{SS}. Almost-sure (a.s.) asymptotic properties of solutions for some classes of non-homogeneous SDEs were considered in \cite{KT} and \cite{BR}.
    
    The idea of determining asymptotic properties of solutions of equations excited by some random process with the solution of an ODE arises naturally. Stochastic models approximate the real processes
    much better than deterministic ones, but are more complicated in the investigation of issues related to asymptotic. By combining powerful methodological tools from two worlds, i.e., the field of ODEs on the one hand and the field of stochastic on the other, researchers solve essential mathematical problems.
    
    Works \cite{GS}, \cite{KKR}, \cite{BKST}, \cite{SS}, \cite{KT} \cite{BR} only focused on the case of SDEs driven by a Brownian motion. Unfortunately, Brownian motion cannot accurately describe stochastic disturbances in certain real systems, such as fluctuations in financial markets, because it is a continuous stochastic process. It is recognized that SDEs with L\'{e}vy noise are more suitable for describing these discontinuous systems, leading to increased interest in SDEs driven by noise with discontinuous jumps.   The case is particularly interesting when the noise is obtained from the  L\'{e}vy process using the  L\'{e}vy-It\^o decomposition into a continuous component perturbed by a Wiener process and a jump-shaped part with a Poisson random measure.  Since the long-time asymptotic behaviour of solutions to SDEs is very important, in this paper we study some asymptotic properties of equations where driving noise has jumps. 
      In  this paper we focus on SDEs driven by L\'{e}vy noise  of the form
      \begin{equation} \label{SDEwJm}
    dX(t)=g(X(t-))\phi(t-)dt+\sigma(X(t-))\theta(t-)dW(t)+\int_{\mathbb{R}}c(X(t-))\gamma(u)\tilde{N}(dt, du), 
      \end{equation}
     and we will search for conditions under which the solution of this equation is asymptotically defined by a non-random function. 
    
     D. Applebaum, M. Siakally \cite{AS} established asymptotic stability of SDEs class (\ref{SDEwJm}) with $\phi=\theta=1$ and $c(x)\gamma(t)=H(x,t).$ The similar problem  for  partial case of SDE  (\ref{SDEwJm})  with numerical approximation was considered in \cite{ChH}.  In \cite{Yus}, was investigated the conditions under which the solution $X$ of the equation from \cite{AS}  has the same asymptotic as corresponds ordinary differential equation. The author proved that in the case when a drift coefficient is  equivalent to some non-random positive degree function then
     $X(t)\sim ((1-\alpha)At)^{\frac{1}{1-\alpha}}$ a.s. as $t\to\infty$ for  $\alpha \in [0;1)$, $A>0.$
     
    In the multidimensional case the problem on identification of a limit of an ODE 
    with discontinuous drift that perturbed by a zero-noise is considered in \cite{PPm}.
     F. Proske and A. Pylypenko \cite{PP} study zero-noise limits of $\alpha -$stable noise perturbed ODE's which are driven by an irregular vector field $A$ with asymptotics $ A(x)\sim \overline{a}\left(\frac{x}{\left\vert x\right\vert }\right)\left\vert x\right\vert ^{\beta -1}$ at zero, where $\overline{a}>0$ is a continuous function and $\beta \in (0,1)$.
    Thus, the issue considered in this paper is a classical subject of stochastic analysis. However the asymptotic long-time 
     properties solution equation (\ref{SDEwJm})  wasn't  investigated yet.
     
     The rest of this paper is organized as follows. In Section 2, we introduce the model, notations and several
    necessary hypotheses. In Section 3, we formulate It\^o's formula and Dynkin's formula for special case SDE with jumps. In Section 4, we present  sufficient condition for a.s. asymptotic equivalence solution of
    SDEs driven by L\'{e}vy noise and solution ODE.  Finally, in Section 5, we
    conclude this paper with some  results with regard to  unboundedness of solutions SDE driven
    by  L\'{e}vy noise.

    \section{Preliminaries}
    Let $(\Omega,\mathfrak{ F}, (\mathfrak{F}_t, t\geq 0), \mathsf{P})$ be a filtered probability space that satisfies the usual hypotheses of completeness and right continuity. Assume that we are given  $\mathfrak{F}_t$-adapted  Wiener process $W=(W(t), t\geq0)$, and    $\mathfrak{F}_t$-adapted Poisson random measure $N$ defined on $\mathbb{R}^{+}\times (\mathbb{R}\setminus \{0\})$. Note that  the process $(W(t), t\geq 0)$ and  $(N(t), t\geq 0)$ are independent.
    
    We set 
    $$\widetilde{N}(dt, du)=N(dt, du)-dt\nu(du),$$
    where $\widetilde{N}$ is the compensated Poisson random measure, $\nu$  is an intensity measure such that $\displaystyle \int_{\mathbb{R}\setminus \{0\}}u^2\nu(du)< \infty.$ 
    
     Denote by $\mathbf{C}$ ($\mathbf{C}^{+}$ ) the class of all continuous (and positive) functions and by $\mathbf{C}^1$ ($\mathbf{C}^{1,
    +}$ ) the class of all continuously differentiable (and positive) functions.
    
     Consider SDEs driven by L\'{e}vy  noise of the form
     \begin{equation} \label{SDEwJmps}
    dX(t)=g(X(t))\phi(t)dt+\sigma(X(t))\theta(t)dW(t)+\int_{|u|<\mathcal{K}}c(X(t-))\gamma(u)\tilde{N}(dt, du), 
      \end{equation}
    on $t\geq 0$ with initial value $X(0)=X_0$ and we also assume that $\underset{t\to \infty}{\lim} X (t)=\infty$ a.s. Sufficient conditions of unboundedness of $X(t)$ is considered in Section 5.
    
    Introduce next assumtions
    \begin{itemize}
     \item[\textbf{I.}] Functions $g(\cdot)\in \mathbf{C}^{+}$, $\sigma(\cdot)\in\mathbf{C}^{+}$, $\phi(\cdot)\in\mathbf{C}^{+}$, and $\theta(\cdot)\in\mathbf{C}$, $c(\cdot)\in \mathbf{C}$, $\gamma(\cdot)\in\mathbf{C}$ are such that:
      \begin{itemize}
        \item[(i)] (Lipschitz condition) for some constants $L_i >0$, $i=1,2,3,$ and all $t\in[0, T]$ and $x,y\in\mathbb{R}$
         \[
           |g(x) - g(y)| \leq L_1|x - y|,\,\,\, |\sigma(x) - \sigma(y)| \leq L_2|x - y|,
         \]
         \[
           |c(x) - c(y)| \leq L_3|x - y|;
         \]
        \item[(ii)] $\displaystyle\int_{\mathbb{R}}\gamma^2(u)\nu(du)<\infty; $
        
        \item[(iii)] $c(\cdot)$ is bounded. 
      \end{itemize}
    \end{itemize}
    \begin{remark}
     Note that, if functions $g$, $\sigma$ and $c$ are continuous and satisfy Lipschitz condition, then they will also satisfy growth condition:
     for some constant $K_i>0,\ i=1,2,3$, and all $t\in[0,T]$ and $x\in\mathbb{R}$
         \begin{equation}\label{GrowthCndn}
           |g(x)|\leq K_1 (1+|x|),\ |\sigma(x)|\leq K_2 (1+|x|),\ |c(x)|\leq K_3 (1+|x|).
         \end{equation}
     It means that if condition \textbf{I} is fulfilled, then the solution of the equation \eqref{SDEwJmps} exists and unique (see, for example \cite{App}).
    \end{remark}
    \begin{remark}
    From \textbf{I.(ii)} and \textbf{I.(iii)} it follows that 
     \[
       \displaystyle\mathsf{E}\left\{ \int_0^t\int_{|u|<\mathcal{K}}|c(X(r-))\gamma(u)|\nu(du)dr\right\}<\infty.
     \]
    Then 
    \begin{align*}
    &\int_0^t\int_{|u|<\mathcal{K}}c(X(t-))\gamma(u)\tilde{N}(dt, du)=\\
    &=\int_0^t\int_{|u|<\mathcal{K}}c(X(t-))\gamma(u)N(dt, du)-\int_0^t\int_{|u|<\mathcal{K}}c(X(r-))\gamma(u)\nu(du)dr   
    \end{align*}   
    \end{remark}
    
    Introduce also the following  assumptions:
    \begin{itemize}
        \item[\textbf{II.}] Function $G(x)=\int\limits_b^x  \frac{ds}{g(s)},\  x\ge b$ is such that
         \[
           \underset{x\to\infty}\lim  G(x)=\infty;
         \]
        \item[\textbf{III.}] Function $\Phi (t)=\int\limits_0^t \phi (u) du>0,\ t>0,$ and 
          \[
            \underset{t\to\infty}\lim \frac{t}{\Phi(t)}  =0.
          \]
    \end{itemize}
    \begin{itemize}
        \item[\textbf{IV.}] Function $\Theta(t)=\int\limits_{0 }^{t}\theta^2 (s)ds,\ t\geq0$ is such that
          \[
            \sum\limits_{k=0}^\infty \frac{\Theta(2^{k+1})}{\Phi^2 (2^k )}<\infty.
          \]
    \end{itemize}
    \begin{remark}
        For the assumption \textbf{IV}, we can provide some simple sufficient conditions  in terms of power functions. It is easy to see, that if $$ \underset{t\to\infty}{\lim\sup\,} \frac{\Theta(t)}{t^{\beta}}>0\,\,\,\textit{and} \,\,\,  \underset{t\to\infty}{\lim\inf\,} \frac{\Phi(t)}{t^{\alpha}}>0$$ are satisfied with $\beta<2\alpha$, then \textbf{IV} holds.
    \end{remark}
    \begin{example}
    It is possible to represent the sufficient
    condition in terms of the regularly varying  functions.
    
    Fix $d\geq 0$ and let $f : [d,\infty) \to (0,\infty)$ be a measurable function. We say
    that $f$ is regularly varying of degree $\alpha \in\mathbb{R}$  if
    $$\underset{x\to \infty}{\lim}\frac{f(cx)}{f(x)}= c^{\alpha},$$
    for all $c > 0$. We will denote the class of regularly varying functions of degree $\alpha$ 
    on $\mathbb{R}^+$ by $\mathscr{RV}_{\alpha}$. Elements of the class $\mathscr{RV}_0$ are said to be \emph{slowly varying}. Examples
    of functions in $\mathscr{RV}_0$ are $log(1+x)$ and $ log( log(e+x))$. Clearly $f(\cdot) \in \mathscr{RV}_{\alpha}$
    if and only if there exists $l(\cdot) \in \mathscr{RV}_0$ such that
    $f(x) = l(x)x^{\alpha}$,
    for all $x \in \mathbb{R}^+$. The following representation theorem for slowly varying
    functions can be found in Bingham et al. (\cite {Bin})
    \begin{remark}
      Assume that $\phi(\cdot)\in\mathscr{RV}_{\alpha}$ and $\theta(\cdot)\in\mathscr{RV}_{\beta}$.
    According to Karamata's Theorem
    $$
    \Phi(\cdot)\in\mathscr{RV}_{\alpha+1}, \qquad \Theta(\cdot)\in\mathscr{RV}_{2\beta+1}.
    $$
    Therefore condition~\textbf{IV} holds if $\alpha\geq\beta$.  
    \end{remark}
        
    \end{example}


    
    \section{The It\^o formula}
    \begin{theorem}\label{Ito}
     Suppose \begin{align}
     X(t)=X_0+\int_{0}^t g(X(r))\phi(r)dr+&\int_{0}^t\sigma(X(r))\theta(r)dW(r)+\nonumber\\
    &+\int_{0}^t\int_{|u|<\mathcal{K}} c(X(r-))\gamma(u)\tilde{N}(dr,du).
    \end{align} 
    Then  we have
    \begin{align*}
    &G(X(t))-G(X_0)=\\
    &=\int_{0}^t\left(\phi(r)-\frac{1}{2}\frac{g'(X(r))\sigma^2(X(r))}{g^2(X(r))}\cdot\theta^2(r)\right)dr+\int_{0}^t\frac{\sigma(X(r))}{g(X(r))}\cdot \theta(r)dW(r)+\\
    &+\int_{0}^t\int_{|u|<\mathcal{K}}\left\{G(X(r-)+c(X(r-))\gamma(u))-G(X(r-))-\frac{c(X(r-))}{g(X(r-))}\cdot \gamma(u)\right\}\nu(du)dr+\\
    &+\int_{0}^t\int_{|u|<\mathcal{K}}\left(G\left(X(r-)+c(X(r-))\gamma(u)\right)-G(X(r-))\right)\tilde{N}(dr,du).
    \end{align*} 
    
     \end{theorem}
    
    \begin{theorem}[Dynkin formula]
     
    \begin{align*}
    &\mathsf{E}[G(X(t))]=\\
    &=G(X_0)+\Phi(t)-\frac{1}{2}\mathsf{E}\left[\int_{0}^t\frac{g'(X(r))\sigma^2(X(r))}{g(X(r))}\cdot\theta^2(r)dr\right]+\\
    &+\mathsf{E}\left[\int_{0}^t\int_{|u|<\mathcal{K}}\left\{G(X(r-)+c(X(r-))\gamma(u))-G(X(r-))-\frac{c(X(r-))\gamma(u)}{g(X(r-))} \right\}\nu(du)dr\right].
    \end{align*} 
    
     \end{theorem}
    \begin{proof}
    The proof follows from the It\^o formula, Proposition 5.1 in D. Applebaum \cite{App}.
    \end{proof}
    
    \section{Almost-sure asymptotic equivalence solution of SDEs driven by L\'{e}vy  noise and solution ODE}
    
    Additionally assume that 
    \begin{itemize}
        \item[\textbf{V.}] $\frac{\sigma}{g }$ is bounded function;
        \item[\textbf{VI.}] $\underset{x\in\mathbb{R}}\inf\,g(x)=k_g>0$ and $\underset{x\to\infty}\lim  g'(x)=0$;
        \item[\textbf{VII.}] $\underset{t \to\infty}{\lim\sup}\, \frac{\theta^2 (t)}{\varphi (t)} <\infty$.
    \end{itemize}
    
    Now we can formulate the main results of this section.
    \begin{theorem}\label{th:AsEq_by_NRF} 
     If conditions \textbf{I -- VII} hold, then
      \[
        \underset{t\to \infty}{\lim}\frac {G(X(t))}{\Phi(t)}=1\ \ \ \textnormal{a.s.}
      \]
    \end{theorem}
    \begin{theorem}\label{th:AsEq_by_SODE}
     Let $g(\cdot)\in \mathbf{C}^{1,+}$, $\sigma(\cdot)\in\mathbf{C}^{1,+}$, $\phi(\cdot)\in\mathbf{C}^{+}$, $\theta(\cdot)\in\mathbf{C}^{+}$, $c(\cdot)\in\mathbf{C}^{+}$, $\gamma(\cdot)\in\mathbf{C}^{+}$ be such that equation \eqref{SDEwJmps}
    has a c\'{a}dl\'{a}g solution $X(\cdot)$ and $\underset{t\to \infty}{\lim} X(t)=\infty$ a.s. If conditions \textbf{I -- VII} hold and
    \begin{itemize}
        \item[\textbf{VIII.}] $\displaystyle\underset{t\to\infty}{\lim \inf} \int_t^{ct}\frac {du}{g(u)G(u)}>0,\,\, \textit{ for all} \,\, c>0$,
    \end{itemize}
      then 
     \[\underset{t\to \infty}{\lim}\frac{X(t)}{\mu(t)}=1 \,\, \text{a.s.,}\]
     where $\mu(\cdot)$ is a solution 
     \begin{equation}\label{ode}
    d\mu(t)=g(\mu(t))\phi(t)dt.     
     \end{equation}
    \end{theorem}
    
    \subsection{Auxiliary assertions}
    
    $\indent$To prove the theorems let us consider some auxiliary statements.
    
    By It\^o formula from Theorem \ref{Ito},  we have  
     \[
      \begin{aligned}
       G&(X(t))=G(X_0)+\int_{0}^t\left(\phi(r)-\frac{1}{2}\frac{g'(X(r))\sigma^2(X(r))}{g^2(X(r))}\cdot\theta^2(r)\right)dr+\\
       &+\int_{0}^t\frac{\sigma(X(r))}{g(X(r))}\cdot \theta(r)dW(r)+\\
       &+\int_{0}^t\int_{|u|<\mathcal{K}}\left\{G(X(r-)+c(X(r-))\gamma(u))-G(X(r-))-\frac{c(X(r-))\gamma(u)}{g(X(r-))} \right\}\nu(du)dr+\\
       &+\int_{0}^t\int_{|u|<\mathcal{K}}\left(G\left(X(r-)+c(X(r-))\gamma(u)\right)-G(X(r-))\right)\tilde{N}(dr,du)=
      \end{aligned} 
     \]
     \begin{equation}\label{ItoRep4GXAA}
        =G(X_0)+\mathfrak{I}_1(t)+\mathfrak{I}_2(t)+\mathfrak{I}_3(t)+\mathfrak{I}_4(t).\hspace*{60mm} 
     \end{equation}
    \begin{lemma}\label{lem:I3/Phi_20}    
     Assume that functions  $g\in \mathbf{C}^{+}$,  $\phi(\cdot)\in\mathbf{C}^{+}$,  $c(\cdot)\in \mathbf{C}$, $\gamma(\cdot)\in\mathbf{C}$ satisfy \textbf{I, III, VI}.
    Then
    $$  \underset{t\to\infty}{\lim\,}\frac{\mathfrak{I}_3(X(t))}{\Phi(t)}= 0 \,\,\, \textit{a.s.}  $$ 
    \end{lemma}
    \begin{proof}
    By Taylor's formula
    $$G(x+c(x)\gamma(u)) = G(x)+G'(x)c(x)\gamma(u)+R(x, x+c(x)\gamma(u)),$$
    where 
    $$R(x, x+c(x)\gamma(u))=\left(\int_0^1(1-\zeta)G''(x+\zeta \cdot c(x)\gamma(u))d\zeta\right)c^2(x)\gamma^2(u).$$
    Then
     \[
      \begin{aligned}
          G(x+c(x)\gamma(u)) -G(x)-&G'(x)c(x)\gamma(u)=\\
          &=c^2(x)\gamma^2(u)\int_0^1(1-\zeta)G''(x+\zeta \cdot c(x)\gamma(u))d\zeta,
      \end{aligned}
     \]
    and we can estimate the integral $\mathfrak{I}_3$:
    \begin{align}
     &|\mathfrak{I}_3(t)|\leq \nonumber\\
      &\leq\frac{1}{2}\int_0^t\int_{|u|<\mathcal{K}}\left(\underset{0\leq\zeta\leq1}{\sup\,}\Big|G''(X(r-)+\zeta \cdot c(X(r-))\gamma(u))\Big|\right)c^2(X(r-))\gamma^2(u)\nu(du)dr\leq \nonumber\\
      & \leq \frac{\underset{x\in\mathbb{R}}\sup\,|g'(x)|}{2k_g^2}\cdot\int_0^t\int_{|u|<\mathcal{K}}c^2(X(r-))\gamma^2(u)\nu(du)dr\le K_{\mathfrak{I}_3}t,\ \text{a.s .}
    \end{align}
    where $K_{\mathfrak{I}_3}=\dfrac1{2k_g^2}\cdot\underset{x\in\mathbb{R}}\sup\,|g'(x)| \cdot\underset{x\in\mathbb{R}}\sup\,c^2(x)\cdot\int_{\mathbb{R}}\gamma^2(u)\nu(du).$

    Since \textbf{III} holds, then $$  \underset{t\to\infty}{\lim\,}\frac{\mathfrak{I}_3(X(t))}{\Phi(t)}= 0 \,\,\, \textit{a.s.}  $$ 
    \end{proof}

    \begin{lemma}\label{L2}
      Assume that $g(\cdot)\in \mathbf{C}^{+}$,  $\phi(\cdot)\in\mathbf{C}^{+}$, and  $c(\cdot)\in \mathbf{C}$, $\gamma(\cdot)\in\mathbf{C}$ satisfy \textbf{I.(iii)} and \textbf{III}. Then 
       \[
         \underset{t\to\infty}{\lim\,}\frac{\mathfrak{I}_4(X(t))}{\Phi(t)}= 0 \,\,\, \textit{a.s.} 
       \]
    \end{lemma}
    \begin{proof}
     The stochastic integral $\mathfrak{I}_4(\cdot)$ is a square integrable martingale (see \cite{kunita}). 
     
     Using the representation
      \[
        G(x+c(x)\gamma(u))- G(x)=c(x)\gamma(u)\int_0^1(1-\zeta)G^{\prime}(x+\zeta \cdot c(x)\gamma(u))d\zeta
      \]
     and by Doob's martingale inequality 
     $$\mathsf{P}\left(
     \underset{k\geq n}{\sup }\frac{|\mathfrak{I}_4(2^k)|}{\Phi(2^k)} >\epsilon\right)\leq\mathsf{P}\left(
     \underset{t\geq 2^n}{\sup } \frac{|\mathfrak{I}_4(t)|}{\Phi(t)}>\epsilon\right)\leq$$
     $$\leq\frac{4}{\epsilon^2}\sum_{k=n}^{\infty} \left(
    \frac{1}{\Phi^2(2^k)}\int_{0}^{2^{k+1}}\int_{|u|<\mathcal{K}}\mathsf{E}\left(\frac{c^2(X(r-))\gamma^2(u)}{\underset{0\leq\zeta\leq1}{\inf\,}(g^2(X(r-)+\zeta \cdot c(X(r-))\gamma(u)))}\right)\nu(du)dr\right)\leq$$
     $$\leq\frac{4}{\epsilon^2k_g^2}\cdot\underset{x\in\mathbb{R}}\sup\,c(x)\cdot\int_{|u|<\mathcal{K}}\gamma^2(u)\nu(du)\cdot\sum_{k=n}^{\infty}\frac{2^{k+1}}{\Phi^2(2^k)}<\infty.$$
    Since $n\to \infty$ the left-side  last inequality tends to $0$. 
    
    Then from
    $\mathsf{P}\left\{\underset{t\to \infty}{\lim \sup }\frac{|\mathfrak{I}_4(t)|}{\Phi(t)}>0\right\}=0$ it follows that $\mathsf{P}\left\{\underset{t\to \infty}{\lim \sup }\frac{|\mathfrak{I}_4(t)|}{\Phi(t)}=0\right\}=1$, which in fact means  $$\mathsf{P}\left\{\underset{t\to \infty}{\lim }\frac{\mathfrak{I}_4(t)}{\Phi(t)}=0\right\}=1.$$
    
    \end{proof}
    
    
    \subsection{The Proofs of Theorems 4.1 and 4.2 }
    
    \textbf{Proof of Theorem 4.1}\ \ \ 
    \begin{proof}
    From \eqref{ItoRep4GXAA} we get
     \[
       \frac{G(X(t))}{\Phi(t)}=\frac{G(X_0)}{\Phi(t)}+\frac{\mathfrak{I}_1(t)}{\Phi(t)}+\frac{ \mathfrak{I}_2(t)}{\Phi(t)}+\frac{ \mathfrak{I}_3(t)}{\Phi(t)}+\frac{ \mathfrak{I}_4(t)}{\Phi(t)}.
     \]
    
    Since
    $$\underset{t\to \infty}{\lim}\frac{1}{\Phi(t)}\left|\int_{0}^t\frac{g'(X(r))\sigma^2(X(r))}{g^2(X(r))}\cdot\theta^2(r)dr\right|\leq\underset{t\to \infty}{\lim}\frac{L}{\Phi(t)}\int_{0}^t\Big|g'(X(r))\Big|\cdot\theta^2(r)dr=0,$$
    where  $L=\underset{x}\sup\,\frac{\sigma^2 (x)}{g^2 (x)} <\infty $.
    Then 
     \begin{equation}\label{I1/Phi_20}
       \underset{t\to \infty}{\lim}\frac{ \mathfrak{I}_1(t)}{\Phi(t)}=1 \,\, \textnormal{a.s.}
     \end{equation}
    Now, continue to prove the stated convergence result and to estimate $ \mathfrak{I}_2(t)$ we will use the Doob's inequality for martingales
    for fixed $\varepsilon>0$ and the convergent series \textbf{IV}
     
     \[
       \begin{aligned}
         \mathsf{P}\left\{\sup_{k \geq n}\frac{\mathfrak{I}_2(2^k)}{\Phi(2^k)}>\varepsilon\right\}&\leq  \mathsf{P}\left\{\sup_{t \geq 2^n}\frac{\mathfrak{I}_2(t)}{\Phi(t)}>\varepsilon\right\}=\\&= \mathsf{P}\left\{\sup_{t \geq 2^n}\frac{1}{\Phi(t)}\left|\int\limits_{0}^t \frac{\sigma(X(r))}{g(X(r))}\cdot \theta(r) dW(r)\right|> \varepsilon\right\} \leq\\
        &\leq\sum_{k=n}^{\infty}\mathsf{P}\left\{\sup_{t \in [2^k;2^{k+1}]}\frac{1}{\Phi(t)}\left|\int\limits_{0}^t \frac{\sigma(X(r))}{g(X(r))}\cdot \theta(r) dW(r)\right|> \varepsilon\right\}\leq\\
        &\leq \frac{4L}{\varepsilon^2}\sum_{k=n}^{\infty}\frac{\Theta(2^{k+1})}{\Phi^2(2^n)} \longrightarrow\ 0,\ \text{as}\ n\to\infty.
       \end{aligned}
      \]
    
    Note also that from the last relation, given that the sequence $\underset{t \geq 2^n}\sup\,\frac{\mathfrak{I}_2(t)}{\Phi(t)},\ n\in\mathbb{N}$, is monotonically decreasing, we also get
      $
        \sup_{t \geq 2^n}\frac{\mathfrak{I}_2(t)}{\Phi(t)}\ \longrightarrow\ 0$,  as $n\to \infty$  a.s.

     Let $t\geq2$ and $n\in\mathbb{N}$ such that $2^n\leq t < 2^{n+1}$. Then
    \[ \begin{aligned}
       &\left \vert\frac{\mathfrak{I}_2(t)}{\Phi(t)}\right\vert \leq\left\vert\frac{\mathfrak{I}_2(t)}{\Phi(t)}-\frac{\mathfrak{I}_2(2^n)}{\Phi(2^n)}\right\vert+\frac{\mathfrak{I}_2(2^n)}{\Phi(2^n)}\leq \sup_{t \in [2^n;2^{n+1}]}\left\vert\frac{\mathfrak{I}_2(t)}{\Phi(t)}-\frac{\mathfrak{I}_2(2^n)}{\Phi(2^n)}\right\vert+\frac{\mathfrak{I}_2(2^n)}{\Phi(2^n)}
      \\&
 \leq \sup_{t \in [2^n;2^{n+1}]} \frac{\Phi(t)-\Phi(2^n)}{\Phi(t)\Phi(2^n)}\cdot \left\vert\mathfrak{I}_2(2^n)\right\vert+\\&+\sup_{t \in [2^n;2^{n+1}]}\frac1{\Phi(t)} \left\vert\int\limits_{2^n}^{t}\frac{\sigma(X(r))}{g(X(r))}\cdot \theta(r) dW(r)\right\vert+\frac{\mathfrak{I}_2(2^n)}{\Phi(2^n)}\leq\\&
    \leq\left(\frac{\Phi(2^{n+1})}{\Phi(2^n)}-1\right)\cdot \frac{\left\vert \mathfrak{I}_2(2^n)\right\vert}{\Phi(t)} +\sup_{t \geq 2^{n}} \frac{\left\vert\mathfrak{I}_2(t)\right\vert}{\Phi(t)}+ \frac{2\mathfrak{I}_2(2^n)}{\Phi(2^n)} \longrightarrow\ 0,\ \text{as}\ n\to\infty \,\, \text{a.s}. \end{aligned}\]
    So, we have 
     \begin{equation}\label{I2/Phi_20}
       \underset{t\to \infty}{\lim}\frac{ \mathfrak{I}_2(t)}{\Phi(t)}=0 \,\, \textnormal{a.s.} 
     \end{equation}
    
    The result of the Theorem \ref{th:AsEq_by_NRF} follows from \eqref{I1/Phi_20}, \eqref{I2/Phi_20} and Lemmas \ref{lem:I3/Phi_20}, \ref{L2}. 
    \end{proof}

    \textbf{Proof of Theorem 4.2}\ \ \ 
    \begin{proof} From the assumptions \textbf{II} and \textbf{VIII} and by the Theorem 2.2.8 
    \cite{BKST-ek} 
     the inverse function $G^{-1}$ for $G$ preserves the asymptotic equivalence. Furthermore $$\Phi(t)=G(\mu(t)),$$ where $\mu$ is a solution $d\mu(t)=g(\mu(t))\phi(t)dt$ and 
       
      \[\underset{t\to \infty}{\lim}\frac{G(X(t))}{\Phi(t)}=\underset{t\to \infty}{\lim}\frac{G^{-1}(G(X(t)))}{G^{-1}(\Phi(t))}=\underset{t\to \infty}{\lim}\frac{X(t)}{\mu(t)}=1 \,\, \text{a.s.}\]
    \end{proof}
    
    \begin{remark}
    In the case when the function $g(\cdot) \in \mathscr{RV}_{\alpha}$, with $\alpha<1$, the condition \textbf{VIII}  holds and theorem 4.1 and 4.2 establishe the conditions for the asymptotic equivalence of the solutions of the SDE with jumps and the corresponding ordinary differential equation.  
    \end{remark}
    
    \section{Unboudedness of the solution of \\non-homogeneous SDEs driven by L\'{e}vy  noise}
    
    In the first part of the paper, we consider the asymptotic equivalence of the solution to a deterministic equation and a stochastic differential equation, in which a Wiener process and a Poisson noise are mixed. The main result is obtained under the assumption \textbf{II}, which yields immediately that the system \eqref{ode} is unbounded as
    $\underset{t\to \infty}{\lim} \mu (t)=\infty.$ One of the basic assumptions of this paper is that the solution of the
    SDE  (\ref{SDEwJmps}) a.s. increases indefinitely and tends to infinity:
    $\mathsf{P}\{\underset{t\to\infty}\lim\,X(t)=\infty\}>0.$
    The conditions of unboundedness of solutions are presented, for example, for autonomous SDRs in \cite{GS}, and for non-autonomous ones  were studied in \cite{KTu}.
    
    Denote
    $\displaystyle B(x)=\int_0^x\frac{dr}{\sigma(r)}.$
    We further assume that
    $\underset{x\to\infty}\lim\,B(x)=\infty$.
    
    In order  to be able to prove main result in this section we need to formulate the law of large number  for martingale.
    \begin{lemma}\label{lem:mart} 
     Let $M(t)$ be a squared integrable martingale. If the series 
      \[
        \sum_{k=1}^{\infty}\frac{1}{2^{k+1}}\mathsf{E}M^2(2^{k+1})<\infty,
     \]
    then 
     \[
       \underset{t\to\infty}{\lim }\frac{M(t)}{\sqrt{2t\ln \ln t}} =0\,\,\textnormal{a.s.}
     \]
    \end{lemma}
    \begin{proof}
    Consider, for any $k \geq 0$ and $\epsilon > 0$, the following two events:
    $B_k\subset C_k,$ where $B_k=\left\{\underset{2^k\leq t\leq 2^{k+1}}{\sup }\frac{|M(t)|}{\sqrt{2t\ln \ln t}} \geq\epsilon\right\}$, $C_k=\left\{\underset{2^k\leq t\leq 2^{k+1}}{\sup }\frac{|M(t)|}{\sqrt{2^k\ln \ln 2^k}} \geq\epsilon\right\}$. So,
     \begin{equation}\label{ineq4Bk}
       \begin{aligned}
        \mathsf{P}\left(B_k\right)\leq\mathsf{P}\left(C_k\right)\leq\mathsf{P}\left( \underset{ t\leq 2^{k+1}}{\sup }|M(t)| \geq \epsilon\sqrt{2^{k+1}\ln \ln 2^k}\right)&\leq\\
        \leq\frac{4\mathsf{E}M^2(2^{k+1})}{\epsilon^22^{k+1}\ln \ln 2^k}&\leq\frac{4\mathsf{E}M^2(2^{k+1})}{\epsilon^22^{k+1}}.
       \end{aligned}  
     \end{equation}   
    
    For any $m \geq 1$ and $\epsilon > 0$
     \[
      \begin{aligned}
        \mathsf{P}\left(\underset{ t\geq 2^{m}}{\sup }\frac{|M(t)|}{\sqrt{2t\ln \ln t}} \geq\epsilon\right) 
        &=\mathsf{P}\left(\bigcup_{m=k}^\infty\,\left\{\underset{2^k\leq t\leq 2^{k+1}}{\sup }\frac{|M(t)|}{\sqrt{2t\ln \ln t}} \geq\epsilon\right\}\right)= \\
        &=\mathsf{P}\left(\bigcup_{m=k}^{\infty} B_k\right)\leq\sum_{k=m}^{\infty}\mathsf{P}\left(
     B_k\right).
      \end{aligned}
     \]
    Using \eqref{ineq4Bk} we obtain
     \[
       \mathsf{P}\left(\underset{ t\geq 2^{m}}{\sup }\frac{|M(t)|}{\sqrt{2t\ln \ln t}} \geq\epsilon\right)\leq\frac{1}{\epsilon^2}\sum_{k=m}^{\infty}\frac{4}{2^{k+1}}\mathsf{E}M^2(2^{k+1}),
     \]
    and therefore  
    $\mathsf{P}\left(
     \underset{t\to\infty}{\lim\sup }\frac{|M(t)|}{\sqrt{2t\ln \ln t}} >0\right)=0.$
     Finally,
     
    $$ \underset{t\to\infty}{\lim }\frac{\vert M(t)\vert }{\sqrt{2t\ln \ln t}} =0\,\,\textnormal{a.s.}$$
    and Lemma \ref{lem:mart}  is proved.
    \end{proof}
    
    
    Consider
    
    \[
    \tilde{a}(t,x)=-\frac{\theta '(t)}{\theta^2(t)} B(x)+ \frac{g(x)\phi(t)}{\sigma (x)\theta (t)}-\frac12\sigma '(x)\theta (t).\]
    and denote
    
      \begin{equation}\label{Q__3_1_}
       A(t)=\int\limits_0^t \underset{x\in \mathbb{R}}\inf\, \tilde{a}(r,x)dr.
     \end{equation}
    
    
    In the following theorem we will give conditions that should be examined in order to see if the perturbed equation is unstable.
    \begin{theorem} \label{thm:BX/thta_2_0} 
     Let $g(\cdot)\in \mathbf{C}^{1,+}$, $\sigma(\cdot)\in\mathbf{C}^{1,+}$, $\phi(\cdot)\in\mathbf{C}^{+}$, $\theta(\cdot)\in\mathbf{C}^{+}$, $c(\cdot)\in\mathbf{C}^{+}$, $\gamma(\cdot)\in\mathbf{C}^{+}$ be such that equation \eqref{SDEwJmps}
    has a c\'{a}dl\'{a}g  solution $X(\cdot)$. Assume that
    \begin{itemize}
     \item[(i)] $\underset{t\to\infty}{\lim \inf}\,\frac{A(t)}{\sqrt{2t\ln \ln t} } >1$;
    \item[(ii)] $\displaystyle\int_{|u|<\mathcal{K}}\frac{1}{\theta^2(u)}(B(x+c(x)\gamma(u))-B(x))^2\nu(du)<K|x|^{q}$ for  $q<0$; 
    \item[(iii)] $\displaystyle\int_0^t \int_{|u|<\mathcal{K}}\frac{c^2(X(r-))\gamma^2(u)}{\theta(u)}\nu(du)dr<\infty$;
    \item[(iv)] $\frac{\sigma'(\cdot)}{\sigma^2(\cdot)}$ is bounded.
    \end{itemize}
        Then 
       \[ \underset{t\to\infty}\lim\,\frac {B(X(t))}{\theta(t)}=\infty \,\,\, \textit{a.s.}\]
    \end{theorem}
    \begin{proof}
    Consider $f(t,x)=\frac{1}{\theta(t)}\cdot B(x)$. Since $B$ is a strictly increasing
    function,  $ B^{-1}(\theta(t)x)$ is the inverse for $f(t,\cdot)$ with respect
    to the argument $x$. Indeed,
    $$f(t, f^{-1}(t,x))=f(t, B^{-1}(\theta(t)x))=\frac{1}{\theta(t)}B(B^{-1}(\theta(t)x))=\frac{1}{\theta(t)}\cdot x\theta(t)=x.$$
    On the other hand
    $$f^{-1}(t,f(t,x))=B^{-1}\left( \frac{1}{\theta(t)}\cdot \theta(t)\cdot B(x)\right)= B^{-1}(B(x))=x.$$
    By It\^o formula with 
    \[
      f^{\prime}_{t}(t,x)=-\frac{\theta^{\prime}(t)}{\theta^2(t)}\cdot B(x),\ \ \  
      f^{\prime}_{x}(t,x)=\frac{1}{\theta(t)}\cdot \frac{1}{\sigma(x)},\ \ \   
      f^{\prime \prime}_{xx}(t,x)=\frac{1}{\theta(t)}\cdot\frac{\sigma^{\prime}}{\sigma^2(x)}
    \]
    we have
    \begin{align}\label{log1}
    & f(t, X(t))=\nonumber\\
    &=f(0, X(0))+W(t)+\int_0^t \left( -\frac{\theta^{\prime}(r)}{\theta^2(r)}\cdot B(X(r))+\frac{g(X(r))\phi(t)}{\sigma(X(r))\theta(r)}-\frac{1}{2}\sigma^{\prime}(X(r))\theta(r)\right)dr +\nonumber\\
    &+\int_0^t \int_{|u|<\mathcal{K}}\left(\frac{1}{\theta(u)}(B(X(r-)+c(X(r-))\gamma(u))-B(X(r-)))-\frac{c(X(r-))\gamma(u)}{\sigma(X(r-))\theta(u)}\right)\nu(du)dr+\nonumber\\
    &+\int_0^t \int_{|u|<\mathcal{K}} \frac{1}{\theta(u)}(B(X(r-)+c(X(r-))\gamma(u))-B(X(r-)))\tilde{N}(dr,du) =\nonumber\\
    &=f(0, X(0))+W(t)+\mathfrak{I}_5(t)+\mathfrak{I}_6(t)+\mathfrak{I}_7(t).
    \end{align}
    
    By assumption \textbf{(ii)} of the Theorem we obtain
    \begin{align*}
    \mathsf{E}\mathfrak{I}_7^2(2^{k+1})&=\\
    &=\mathsf{E}\int_0^{2^k} \int_{|u|<K}\frac{1}{\theta^2(u)}(B(X(r-)+c(X(r-))\gamma(u))-B(X(r-)))^2\nu(du)dr\leq\\
    &\leq\int_0^{2^k}K|r|^qdr=\frac{K2^{k(q+1)}}{q+1}.
    \end{align*}
    Using Lemma \ref{lem:mart} for $M(t)=\mathfrak{I}_7(t)$,
    $$\sum_{k=1}^{\infty}\frac{1}{2^{k+1}}\mathsf{E}\mathfrak{I}_7^2(2^{k+1})=\frac{K}{2(q+1)}\sum_{k=1}^{\infty}2^{kq}<\infty,$$ 
    as $q<0$, we get
     \[
       \underset{t\to\infty}{\lim }\frac{\mathfrak{I}_7(t)}{\sqrt{2t\ln \ln t}} =0\,\,\textnormal{a.s.}
     \]
    By assumptions \textbf{(iii)} and \textbf{(iv)},
    \begin{align*}
    &\mathfrak{I}_6(t)=\\
    &\int_0^t \int_{|u|<\mathcal{K}}\left(\frac{(B(X(r-)+c(X(r-))\gamma(u))-B(X(r-)))}{\theta(u)}-\frac{c(X(r-))\gamma(u)}{\sigma(X(r-))\theta(u)}\right)\nu(du)dr\leq\\
    &\leq\frac{1}{2}\int_0^t \int_{|u|<\mathcal{K}}\left(\underset{0\leq\zeta\leq1}{\sup\,}\Big|B''(X(r-)+\zeta \cdot c(X(r-))\gamma(u))\Big|\right)\frac{c^2(X(r-))\gamma^2(u)}{\theta(u)}\nu(du)dr\leq \\
      & \leq \frac{1}{2}\underset{x\in\mathbb{R}}{\sup\,\,} \underset{0< |u|\leq \mathcal{K}}{\sup\,}|B''(x,u)| \int_0^t \int_{|u|<\mathcal{K}}\frac{c^2(X(r-))\gamma^2(u)}{\theta(u)}\nu(du)dr<\infty.
    \end{align*}
    
    Therefore
     \[
       \underset{t\to\infty}{\lim }\frac{\mathfrak{I}_6(t)}{\sqrt{2t\ln \ln t}} =0\,\,\textnormal{a.s.}
     \]
    Let $\varepsilon(\cdot)$ be some positive function, such that $\int_0^{\infty}\varepsilon(r)dr<\infty$. Since $$\underset{x\in \mathbb{R}}\inf\, \tilde{a}(t,x)-\varepsilon(t)<\tilde{a}(t,x),$$ then
    $f(t,X(t))>f(0, X(0))+W(t)+\int_0^t(\underset{x\in \mathbb{R}}\inf\, \tilde{a}(r,x)-\varepsilon(r))dr+\mathfrak{I}_6(t)+\mathfrak{I}_7(t)$.
    
    By the
    low of large number for Wiener process we have
    \[
     \begin{aligned}
      \underset{t\to\infty}{\lim\inf }\frac{f(t, X(t))}{\sqrt{t\ln \ln t}}>\underset{t\to\infty}{\lim\inf }\frac{f(0, X(0))}{\sqrt{t\ln \ln t}}&+\underset{t\to\infty}{\lim\inf }\frac{W(t)}{\sqrt{2t\ln \ln t}}+\\
      &+\underset{t\to\infty}{\lim\inf }\frac{1}{\sqrt{2t\ln \ln t}}\left(A(t)-\int_0^t\varepsilon(r)dr\right)+\\
      &+\underset{t\to\infty}{\lim\inf }\frac{\mathfrak{I}_6(t)}{\sqrt{2t\ln \ln t}}+\underset{t\to\infty}{\lim\inf }\frac{\mathfrak{I}_7(t)}{\sqrt{2t\ln \ln t}}.
     \end{aligned}
    \]
    From the obtained inequality follows the statement of the Theorem \ref{thm:BX/thta_2_0}.
    \end{proof}
    
    \begin{theorem}
     If all conditions of Theorem \ref{thm:BX/thta_2_0} hold and $\underset{t\to\infty}\lim \theta(t)=\infty$, then  
      \[
        \underset{t\to\infty}\lim X(t)=\infty\ \text{a.s.}
      \]
    \end{theorem}
    \begin{proof}
    The result of the Theorem follows from the Theorem \ref{thm:BX/thta_2_0}
    \end{proof}
  
    \section{Conclusion}\label{sec7}

To conclude, the purpose of this paper is to investigate the asymptotic properties of SDEs perturbed by a Wiener process
and a jump-shaped part with a Poisson random measure. We have focused on SDEs with coefficients expressed as  the product of
two functions: one  non-random and depending on the time, the other depending on a random process. We have:
\begin{itemize}
\item  Proposed a method for studying the asymptotic properties of a solution of a solution to an SDE by comparing it with a solution to an ODE obtained by removing the stochastic part;
\item Proved that under certain  conditions  the asymptotic properties of  SDEs solutions are determined by a nonrandom function.
\item  Provided  results regarding the unboundedness of solutions for
SDEs with time-dependent coefficients  as $t\to \infty$.
\end{itemize}
    
We plan to further our research in order to extend our investigation of the  asymptotic behavior of solutions of SDEs to the multidimensional case. 
 
 \section*{Statements and Declarations}

\begin{itemize}
\item \textbf{Funding.} The article was written within the project "Robustness, Asymptotic Behavior and Stability of
Models Based on Stochastic Differential Equations". This research was supported by  the MSC4Ukraine grant (AvH ID:1233636), which is funded by the European Union.

We would like to thank the Alexander von Humboldt Foundation
 for supporting Ukrainian scientists and their research.

\item \textbf{Conflict of interest.} The authors have no competing interests or other interests that might be perceived to influence the results and/or discussion reported in this paper.
\item \textbf{Authors' contributions.} All authors contributed to of the research concept.
Material preparation and analysis were performed
by all authors.   The first draft of the manuscript was written by O. Tymoshenko.  All authors read and approved the final manuscript.
\end{itemize}
    
    
    

\begin{thebibliography}{00}
    
    \bibitem{App}  
    D. Applebaum (2004),\emph{ L\'{e}vy Processes and Stochastic Calculus}, 2nd edn, Cambridge University Press.
    
    \bibitem{AS} 
     D. Applebaum and M. Siakalli (2009),\emph{ Asymptotic stability of stochastic differential equations driven by Levy noise}, J. Appl. Probab. 46 , no. 4, 1116-1129 https://doi.org/10.1239/jap/1261670692.
    
    \bibitem{Bin} 
    N. H. Bingham, C. M. Goldie, J. L. Teugels (1987), \emph{Regular Variation}, Cambridge University Press.
    
    \bibitem{GS}  
    I.I. Gikhman, A.V. Skorokhod (1972), \emph{Stochastic Differential Equations}, Springer, Berlin.
    
    
    \bibitem{kunita}  
    H. Kunita (2019), \emph{ Stochastic Flows and Jump-Diffusions}, Springer Singapore.
    
    \bibitem{KKR}  
    G. Keller, G. Kersting, and U. Rosler (1984), \emph{On the asymptotic behaviour of solutions of stochastic
    differential equations}, Z. Wahrsch. verw. Geb. 68, 163-184.
    
    
    \bibitem{BKST}  
    V.V. Budygin, O.I. Klesov, J.G. Steinebach, and O.A. Tymoshenko (2008), \emph{On the $\phi$ -asymptotic
    behaviour of solutions of stochastic differential equations}, Theory Stoch. Process. 14,
    11-30.
    
    \bibitem{BKST-ek} 
    V. V. Buldygin, K.-H. Indlekofer, O. I. Klesov, and J. G. Steinebach (2018), \emph{ Asymptotic behavior of solutions of stochastic differential equations. PseudoRegularly Varying Functions and Generalized Renewal Processes}. 
    
    \bibitem{SS} 
    A.M. Samo\v{i}lenkolenko, O.M. Stanzhytskyi (2011), \emph{Qualitative and Asymptotic Analysis of Differential Equations with Random Perturbations}, World Scientific Publishing, Hackensack, NJ .
    
    \bibitem{KT}
    O. I. Klesov,  O. A. Tymoshenko (2019),   \emph{Almost sure asymptotic properties of solutions of a class of non-homogeneous stochastic differential equations.} Modern Mathematics and Mechanics. Fundamentals Problems and Challenges. Switzerland: Springer, 97-114.
    https://doi.org/10.1007/978-3-319-96755-46.
    
    \bibitem{BR} 
    G. Berkolaiko and A. Rodkina (2006), \emph{Almost sure convergence of solutions to non-homogeneous stochastic difference equation}, J. Differ. Equations Appl. 12, pp. 535-553. 
    
    \bibitem{PP} 
    A. Pilipenko, F. Proske (2021), \emph{Small noise perturbations in multidimensional case}, https://doi.org/10.48550/arXiv.2106.09935. 
    
    \bibitem{PPm} 
    A. Pilipenko and F. N. Proske (2018), \emph{ On a selection problem for small noise
    perturbation in the multidimensional case}. Stochastics and Dynamics, 18 (06):1850045.
    
    \bibitem{Yus} 
    V.  Yuskovych (2023), \emph{ On asymptotics of solutions of stochastic defferential equations
    with jumps}, Ukrains'kyi Matematychnyi Zhurnal, Vol. 75, no. 11,  2023, pp. 1570-84, doi:10.3842/umzh.v75i11.7684.
    
    \bibitem{ChH} 
    G. Chalmers, D. Higham (2008), \emph{Asymptotic stability of a jump-diffusion equation and its numerical approximation}, SIAM J. Sci. Comput.  Vol. 31, No. 2, pp. 1141-1155.
    
    \bibitem{KTu}
    Klesov, O.I., Tymoshenko, O.A. (2013), \emph{ Unbounded solutions of stochastic differential equations with time-dependent coefficients}. Ann. Univ. Sci. Budapest Sect. Comput. 41, 25-35.
    
    \end{thebibliography}
    

    \end{document}